\documentclass[12pt]{amsart}
\usepackage{amssymb,amscd}

\let\frak\mathfrak
\let\Bbb\mathbb

\def\>{\relax\ifmmode\mskip.666667\thinmuskip\relax\else\kern.111111em\fi}
\def\<{\relax\ifmmode\mskip-.333333\thinmuskip\relax\else\kern-.0555556em\fi}
\def\vsk#1>{\vskip#1\baselineskip}
\def\vv#1>{\vadjust{\vsk#1>}\ignorespaces}
\def\vvn#1>{\vadjust{\nobreak\vsk#1>\nobreak}\ignorespaces}

\let\Medskip\medskip
\def\medskip{\par\Medskip}
\let\Bigskip\bigskip
\def\bigskip{\par\Bigskip}

\let\Maketitle\maketitle
\def\maketitle{\Maketitle\thispagestyle{empty}\let\maketitle\empty}

\newtheorem{thm}{Theorem}[section]

\newtheorem{lem}[thm]{Lemma}

\newtheorem{conj}[thm]{Conjecture}

\numberwithin{equation}{section}

\theoremstyle{definition}
\newtheorem*{rem}{Remark}

\let\mc\mathcal
\let\nc\newcommand

\nc{\on}{\operatorname}
\nc{\Z}{{\mathbb Z}}
\nc{\C}{{\mathbb C}}
\nc{\N}{{\mathbb N}}
\nc{\pone}{{\mathbb C}{\mathbb P}^1}
\nc{\arr}{\rightarrow}
\nc{\larr}{\longrightarrow}
\nc{\al}{\alpha}
\nc{\W}{{\mc W}}
\nc{\la}{\lambda}
\nc{\su}{\widehat{{\mathfrak sl}}_2}
\nc{\g}{{\mathfrak g}}
\nc{\h}{{\mathfrak h}}
\nc{\m}{{\mathfrak m}}
\nc{\n}{{\mathfrak n}}
\nc{\Gm}{\Gamma}
\nc{\La}{\Lambda}
\nc{\gl}{\widehat{\mathfrak{gl}_2}}
\nc{\bi}{\bibitem}
\nc{\om}{\omega}
\nc{\Res}{\on{Res}}
\nc{\gm}{\gamma}
\nc{\Om}{\Omega}
\nc{\yy}{{\bs y}}
\nc{\kk}{{\bs k}}

\def\Res{\on{Res}}

\def\B{{\mc B}}

\def\F{{\mc F}}

\def\V{{\mc V}}

\let\Dl\Delta

\let\Si\Sigma

\let\geq\geqslant

\let\leq\leqslant

\nc{\gln}{\mathfrak{gl}_N}
\nc{\sln}{\mathfrak{sl}_N}

\def\beq{\begin{equation}}
\def\eeq{\end{equation}}
\def\be{\begin{equation*}}
\def\ee{\end{equation*}}

\nc{\bean}{\begin{eqnarray}}
\nc{\eean}{\end{eqnarray}}
\nc{\bea}{\begin{eqnarray*}}
\nc{\eea}{\end{eqnarray*}}
\nc{\bs}{\boldsymbol}
\nc{\Ref}[1]{{\rm(\ref{#1})}}

\nc{\R}{\Bbb R}
\nc{\glN}{\mathfrak{gl}_N}
\nc{\glNt}{\mathfrak{gl}_N[t]}
\nc{\s}{sing}

\nc{\Oml}{{\Om_{\bs\la}}}
\nc{\OmLb}{{\Om_{\bs\La,\bs\la,\bs b}}}
\nc{\Ol}{{\mc O_{\bs\la}}}
\nc{\OLb}{{\mc O_{\bs\La,\bs\la,\bs b}}}
\nc{\Ml}{{\mc M_{\bs\la}}}
\nc{\Mlb}{{\mc M_{\bs\La,\bs\la,\bs b}}}
\nc{\Blb}{{\B_{\bs\La,\bs\la,\bs b}}}
\nc{\Omn}{{\Omega_{\bs n,\bs b,\bs K}}}
\nc{\Omlb}{{\bar\Om_{\bs\la}}}
\nc{\VSl}{{(\V^S)_{\bs\la}}}

\nc{\Dlb}{\Dl_{\bs\La,\bs\la,\bs b,\bs K}}
\nc{\ep}{\epsilon}
\nc{\Vn}{{V^{\otimes n}}}
\nc{\Il}{{\mc I_{\bs\la}}}
\nc{\bla}{{\bs\la}}
\nc{\Fla}{\F_{\bs\la}}
\nc{\GL}{{GL_n(\C)}}
\nc{\ga}{\gamma}
\nc{\Ga}{\Gamma}

\nc{\Nn}{{\mc N}}
\nc\Ll{{\mc L}}

\nc{\PCN}{{   (\C[x])^2   }}
\nc{\slt}{{\frak{sl}_3}}
\nc\ad{{\on{ad}}}
\nc\gA{{\g(A_2^{(1)})}}
\nc\At{{A_2^{(1)}}}
\nc\Dia{{\on{Diag}}}
\nc\8{\emptyset}
\nc\AT{{A^{(2)}_2}}
\nc{\PCr}{{ \bs P  (\C[x])^r   }}

\begin{document}
\title[On the number of populations of critical points]
{On the number of populations of critical points of master functions}

\author[{}]{ Evgeny Mukhin
${}^{*,1}$
\and Alexander Varchenko${}^{**,2}$}
\thanks{${}^1$ Supported in part by NSF grant DMS-0900984}

\thanks{${}^2$ Supported in part by NSF grant DMS-1101508}

\maketitle
\medskip \centerline{\it ${}^*$
Department of Mathematical Sciences,}
\centerline{\it Indiana University
Purdue University Indianapolis,}
\centerline{\it 402 North Blackford St., Indianapolis,
IN 46202-3216, USA}
 \medskip
\centerline{\it ${}^{**}$Department of Mathematics, University of
  North Carolina at Chapel Hill,} \centerline{\it Chapel Hill, NC
  27599-3250, USA} \medskip

\begin{abstract}
We consider the master functions associated with one irreducible integrable highest weight representation of a Kac-Moody algebra.
We study the generation procedure of new critical points from a given critical point of one of these master functions.
We show that all critical points of all these master functions can be generated from
the critical point of the master function with no variables. In particular this means that the set of all critical points of all these master functions
form  a single population of critical points.

We formulate a conjecture that  the number of populations of critical points of master functions associated with
a tensor product of irreducible integrable highest weight representations of a Kac-Moody algebra are labeled by homomorphisms to $\C$
of the Bethe algebra of the Gaudin model associated with this tensor product.

\end{abstract}



\section{Introduction}
We consider the master functions associated with one irreducible integrable highest weight representation of a Kac-Moody algebra.
We study the generation procedure of new critical points from a given critical point of one of these master functions.
We show that all critical points of all these master functions can be generated from
the critical point of the master function with no variables. In particular this means that the set of all critical points of all these master functions
form  a single population of critical points.

We formulate a conjecture that  the number of populations of critical points of master functions associated with
a tensor product of irreducible integrable highest weight representations of a Kac-Moody algebra are labeled by homomorphisms to $\C$
of the Bethe algebra of the Gaudin model associated with this tensor product.

In Section \ref{crit point sec} we introduce master functions and describe different ways to characterize critical points of master functions.
In Section \ref{Populations of critical points}, we introduce populations of critical points and formulate the conjecture. The main results are
Theorems \ref{thm c=0} and \ref{thm n=1}.

The second author thanks the Max Planck Institute for Mathematics in Bonn for hospitality.

\section{Master functions and critical points, \cite{MV1}}\label{crit pts}
\label{crit point sec}

\subsection{Kac-Moody algebras}
\label{Kac_Moody sec}
Let $A=(a_{i,j})_{i,j=1}^r$ be a generalized  Cartan matrix,
$a_{i,i}=2$,
$a_{i,j}=0$ if and only $a_{j,i}=0$,
 $a_{i,j}\in \Z_{\leq 0}$ if $i\ne j$.
We  assume that $A$ is symmetrizable, i.e.
there exists a diagonal matrix $B=\on{diag}(b_1,\dots,b_r)$
with positive integers $b_i$ such that $BA$
is symmetric.

Let $\g=\g(A)$ be the corresponding complex Kac-Moody
Lie algebra (see \cite[Section 1.2]{K}),
$\h \subset \g$  the Cartan subalgebra.
The associated symmetric bilinear form  $(\,,\,)$ is nondegenerate on $\h^*$ and
$\dim \h = r + 2d$, where $d$ is the dimension of the kernel of the Cartan matrix $A$,
see \cite[Chapter 2]{K}.

Let $\al_i\in \h^*$, $\al_i^\vee\in \h$, $i = 1, \dots , r$, be the sets of simple roots,
coroots, respectively. We have
$ (\al_i,\al_j)=  b_i a_{i,j},$  $\langle\la ,\al^\vee_i\rangle =
2(\la,\al_i)/{(\al_i,\al_i)}$,\ $ \la\in\h^*.$
In particular, $\langle\al_j ,\al^\vee_i \rangle = a_{i,j}$.
Let $ P = \{ \lambda \in \h^* \, |\, \langle\la, \al^\vee_i\rangle \in \Z\}$ and
$P^+ = \{ \lambda \in \h^* \, |\, \langle\la
,\al^\vee_i\rangle \in \Z_{\geq 0}\}$ be the sets of integral and
dominant integral weights, respectively.

Fix $\rho\in\h^*$ such that $\langle\rho,\al_i^\vee\rangle=1$,
$i=1,\dots,r$. We have $(\rho,\al_i)= (\al_i,\al_i)/2$.
The Weyl group $\mathcal W\in\on{End (\h^*)}$ is generated by
reflections $s_i$, $i=1,\dots,r$, where
$s_i(\la)=\la-\langle\la,\al_i^\vee\rangle\al_i,$ $ \la\in\h^*.$
We use the notation
\bean
\label{shifted}
w\cdot\la=w(\la+\rho)-\rho,\qquad w\in \mathcal W,\;\la\in\h^*,
\eean
for the {\it shifted action} of the Weyl group.


\subsection{Master functions, \cite{SV}}
\label{master sec}
Let $\bs\La=(\La_1, \dots , \La_n)$, $\La_a\in P^+$ be
a collection of dominant integral weights.
Let $\bs k=(k_1,\dots,k_r)\in\Z^r_{\geq 0}$ be
a collection of nonnegative integers. Denote $k=k_1+\dots+k_r$.
Denote $\La_\infty(\bs\La,\bs k) = \sum_{a=1}^n \La_a  -  \sum_{j=1}^r
k_j\al_j \in  P$.
The weight $\La_\infty(\bs\La,\bs k)$ will be called the {\it weight at infinity}.

Consider $\C^n$ with coordinates $\bs z=(z_1,\dots,z_n)$.
Consider $\C^k$ with coordinates $\bs u$ collected into $r$ groups, the $j$-th group consisting of $k_j$ variables,
$\bs u=(u^{(1)},\dots,u^{(r)}),$ $u^{(j)} = (u^{(j)}_1,\dots,u^{(j)}_{k_j})$.
The {\it master function} is the multivalued function on $\C^k\times\C^n$ defined by the formula
\bean
\label{MA general}
&&
\Phi(\bs u,\bs z,\bs\La,\bs k) = \sum_{a<b} (\La_a,\La_b) \ln (z_a-z_b)
-  \sum_{a,i,j} (\al_j,\La_a)\ln (u^{(j)}_i-z_a) +
\\
&&
\phantom{aaa}
+ \sum_{j< j'} \sum_{i,i'} (\al_j,\al_{j'})
\ln (u^{(j)}_i-u^{(j')}_{i'})
+  \sum_{j} \sum_{i<i'} (\al_j,\al_{j})
\ln (u^{(j)}_i-u^{(j)}_{i'}),
\notag
\eean
with singularities at the places where the arguments of the logarithms are equal to zero.
The product of symmetric groups
$\Sigma_{\bs k}=\Si_{k_1}\times\dots\times \Si_{k_r}$ acts on the set of variables $u^{(j)}_i$
by permuting the coordinates with the same upper index.
The  function $\Phi$ is symmetric with
respect to the $\Si_{\bs k}$-action.

For $q\in\C^n$, a point $p\in\C^k$ is a {\it critical point}
if $\Phi(\cdot,z(q),\bs\La,\bs k)$ is defined at $p$ and $d_u\Phi(\cdot,z(q),\bs\La,\bs k)=0$ at $p$.
In other words, $p$ is a critical point if
\bean
\label{Bethe eqn 1}
\sum_{i',\, i' \neq i} \frac {(\al_j,\al_j)}{ u_i^{(j)}(p) - u_{i'}^{(j)}(p)}
+ \sum_{j',i'} \frac {(\al_j,\al_{j'})}{ u_i^{(j)}(p) - u_{i'}^{(j')}(p)}-\sum_{a} \frac {(\al_j,\La_a)}{ u_i^{(j)}(p) - z_a(q)} = 0,
\eean
$j=1,\dots,r,\,i=1,\dots,k_j$.
The critical set of the function $\Phi(\cdot,\bs z(q),\bs\La,\bs k)$ is $\Si_{\bs k}$-invariant.
All orbits in the critical set have the same cardinality $k_1!\dots k_r!$ .
We do not make distinction between critical points in the same orbit.

\subsection{ Polynomials representing critical points }
\label{PLCP}

Let $q\in\C^n$.
Let $p\in\C^k$ be a critical point of the master function
$\Phi(\cdot,\bs z(q), \bs\La,\bs k)$.
Introduce an $r$-tuple of polynomials $\bs y=( y_1(x),$ $\dots ,$ $ y_r(x))$,
$y_j(x) = \prod_{i=1}^{k_j}(x-u_i^{(j)}(p))$.
Each polynomial is considered up to multiplication
by a nonzero number.
The tuple defines a point in the direct product
$\PCr$ of $r$ copies of the projective space associated with the vector
space of polynomials in $x$.
We say that the tuple $\bs y$ {\it represents the
critical point} $p$.
The vector $\bs k=(k_1,\dots,k_r)$ will be called the {\it degree vector} of the tuple $\bs y$.

It is convenient to think that the tuple $(1, \dots , 1)$ of constant polynomials
represents  in $\PCr$ the critical point of
the master function with no variables.
This corresponds to the degree vector $\bs k = (0, \dots , 0)$.

Introduce polynomials
\bean
\label{T}
{T}_j(x)=\prod_{a=1}^n(x-z_a(q))^{\langle \La_a, \al_j^\vee\rangle}, \qquad j = 1, \dots , r .
\eean
Denote $\tau_j=\deg T_j(x)$.
We say that a given tuple $\bs y\in\PCr$ is {\it generic} with respect
to a point $q\in\C^n$ and weights $\bs\La$
if: \ (i) each polynomial $y_j(x)$ has no multiple roots;\ (ii) all roots of $y_j(x)$ are different from roots of the polynomial $T_j$;
\ (iii) any two polynomials $y_i(x)$, $y_j(x)$ have no common roots if $i\neq j$ and
$a_{i,j}\neq 0$.
It is clear that if a tuple represents a critical point, then it is generic, see equations \Ref{Bethe eqn 1}.

Denote $W(f,g)=fg'-f'g$  the Wronskian determinant of functions $f, g$ in $x$.
A tuple $\bs y$  is called {\it fertile} with respect to weights $\bs\La$ and a point $q\in\C^n$,
  if there exist polynomials $\tilde y_1(x),\dots,\tilde y_r(x)$  satisfying the equations
\bean
\label{wronskian-critical eqn}
W( y_j, \tilde y_j) =   T_j  \prod_{i, \,i\neq j}y_i^{- a_{j,i}}, \qquad j=1,\dots, r,
\eean
see \cite{MV1}.
Equation \Ref{wronskian-critical eqn}  is a first order linear inhomogeneous differential equation with respect to $\tilde y_j$.
Its solutions are
\bean
\label{deG 0}
\tilde y_j = y_j \int T_j \prod_{i=1}^r y_i^{- a_{j,i}}  dx + c y_j,
\eean
where $c$ is any number.
The tuples
\bean
\label{simple 0}
\bs y^{(j)}(x,c) = (y_1(x),\dots,y_{j-1}(x),\tilde y_j(x,c), y_{j+1}(x),\dots,y_r(x))
 \in   \PCr \
\eean
 form a one-parameter family.  This family  is called
 the {\it generation  of tuples  from $\bs y$ in the $j$-th direction}.
 A tuple  of this family is called an {\it immediate descendant} of $\bs y$ in the $j$-th direction.

\begin{thm}
[\cite{MV1}]
\label{fertile cor}
${}$

\begin{enumerate}
\item[(i)]
A generic tuple $\bs y = (y_1,\dots, y_r)$
represents a critical point of the master function $\Phi(\cdot,\bs z(q), \bs\La,\bs k)$, where
$k_j=\deg y_j$,
if and only if $\bs y$ is fertile with respect to  weights  $\bs\La$  and a point
$q\in\C^n$

\item[(ii)] If $\bs y$ represents a critical point,
then for any $c\in\C$ the tuples $\bs y^{(j)}(x,c)$, $j=1,\dots,r$,  are fertile.

\item[(iii)]
If $\bs y$ is generic and fertile, then for almost all values of the parameter
 $c\in \C$ the tuples $\bs y^{(j)}(x,c)$, $j=1,\dots,r$,  are  generic.
The exceptions form a finite set in $\C$.

\item[(iv)]

Assume that a sequence ${\bs y}_i, i = 1,2,\dots$, of fertile tuples
has a limit ${\bs y}_\infty$ in $\PCr$ as $i$ tends to infinity.
\begin{enumerate}
\item[(a)]
Then the limiting tuple ${\bs y}_\infty$ is fertile.
\item[(b)]
For $j = 1,\dots,r$, let ${\bs y}_\infty^{(j)}$ be an immediate
descendant of ${\bs y}_\infty$.
Then for $j=1,\dots,r$, there exist immediate descendants
${\bs y}_i^{(j)}$ of $y_i$ such that ${\bs y}_\infty^{(j)}$
is the limit of ${\bs y}_i^{(j)}$ as $i$ tends to infinity.

\end{enumerate}

\end{enumerate}
\end{thm}

Consider a generic fertile tuple $\bs y=(y_1(x),\dots,y_r(x))$ as in Theorem \ref{fertile cor}. Let $k_j=\deg y_j$
for $j=1,\dots,r$.
Consider a generic fertile tuple $\bs y^{(j)}(x,c)= (y_1(x),\dots,y_{j-1}(x),$ $\tilde y_j(x,c),$ $ y_{j+1}(x),\dots,y_r(x))$
for some $c\in \C$ and some $j,\, 1\leq j\leq r$, as in part (iii) of Theorem \ref{fertile cor}. It is easy to see that the  polynomial
$\tilde y_j(x,c)$ is of degree $k_j$ or $\tilde k_j = \tau_j + 1-k_j - \sum_{i\ne j} a_{j,i}k_i$.
Denote
\bea
\bs k^{(j)} = (k_1,\dots,k_{j-1},\tilde k_j,k_{j+1}\dots,k_r).
\eea
By Theorem  \ref{fertile cor}, if $\deg \tilde y_j(x,c) = k_j$, then the generic fertile tuple  $\bs y^{(j)}(x,c)$
represents a critical point
of the master function $\Phi(\cdot,\bs z(q), \bs\La, \bs k)$.  By Theorem  \ref{fertile cor},
if $\deg \tilde y_j(x,c) = \tilde k_j$, then the generic fertile tuple  $\bs y^{(j)}(x,c)$ represents a critical point
of the master function $\Phi(\cdot,\bs z(q),$ $ \bs\La, \bs k^{(j)})$.

The transformation from $\bs k$ to $\bs k^{(j)}$ can be described in terms of the shifted Weyl group action.

\begin{lem}
[\cite{MV1}]
We have $\La_\infty(\bs \La,\bs k^{(j)}) = s_j\cdot \La_\infty(\bs\La,\bs k)$.
\end{lem}

\subsection{Another way to characterize critical points}

\begin{thm}
[\cite{MSTV}]
\label{new BAE}
Let $\bs y=(y_1,\dots,y_r)$ be generic   with respect
to a point $q\in\C^n$ and weights $\bs\La$. Then $\bs y$
 represents a critical point of the master function
$\Phi(\cdot,\bs z(q); \bs\La,\bs k)$
if and only if there exist numbers $\mu_1,\dots,\mu_n$, $\mu_1+\dots+\mu_n=0$, such that
\bean
\label{new BAE}
&&
\sum_{j=1}^r(\al_j,\al_j)\frac{y_j''}{y_j} + \sum_{i\ne j}(\al_i,\al_j)\frac{y_i'y_j'}{y_iy_j}
-\sum_{j=1}^r(\al_j,\al_j)\frac{T'_jy_j'}{T_jy_j}+
\\
\notag
&&
\phantom{aaaaaaaaaaaaaaaaaa}
+ \sum_{a=1}^n\frac{1}{x-z_a(q)}(\mu_a - \sum_{b\ne a}\frac{(\La_a,\La_b)}{z_a(q)-z_b(q)}) = 0 .
\eean

\end{thm}

\begin{lem}
[\cite{MSTV}]
\label{new desc}
Let $\bs y=(y_1,\dots,y_r)$ be generic  and fertile with respect
to a point $q\in\C^n$ and weights $\bs\La$. Let $j\in\{1,\dots,r\}$.
Let $\bs y^{(j)}(x,c)=(y_1(x),\dots,\tilde y_j(x,c),\dots,y_r(x))$ be
a generic fertile descendant in the $j$-th direction.
Then the tuple  $\bs y^{(j)}(x,c)$ satisfies equation \Ref{new BAE} with the same numbers $\mu_1,\dots,\mu_n$
as the tuple $\bs y$.

\end{lem}

\section{Populations of critical points}
\label{Populations of critical points}

\subsection{Populations}

 Let $\bs y=(y_1(x),\dots,y_r(x))\in\PCr$ be the tuple representing
a critical point $p\in\C^k$ of the master function $\Phi(\cdot,\bs z(q), \bs\La, \bs k)$.
By Theorem  \ref{fertile cor}, we can construct $r$ one-parameter families
$\bs y^{(j)}(x,c)\in\PCr$ of fertile tuples almost all of which are generic with respect to the point $q\in\C^n$ and weights $\bs \La$
and hence represent critical points of master functions.
After that we can start with  any tuple $\bs y^{(j)}(x,c)$ of these $r$ families and  generate new $r$ one-parameter
families of fertile tuples by using Theorem \ref{fertile cor}. This two-step procedure gives us
 $r^2$ two-parameter families of fertile tuples in $\PCr$ almost all of which
are generic with respect to the point $q\in\C^n$ and weights $\bs \La$
and hence represent critical points of master functions. We may repeat this procedure any number of times and after, say, $m$ repetitions
 we will obtain $r^m$ families of fertile tuples almost all of which are generic with respect to the point $q\in\C^n$ and weights $\bs \La$
 and hence represent critical points of master functions.
The union in $\PCr$ of all tuples obtained after all possible repetitions is called the {\it population of tuples generated from the tuple}
$\bs y$ with given additional data $q, \bs \La$ (or called the {\it population of critical points generated from the critical point} $p$
of the master function $\Phi(\cdot,\bs z(q); \bs\La, \bs k)$).  All tuples in the population are fertile with respect to the data $q,\bs\La$.
Almost all tuples are generic with respect to the data $q,\bs\La$.

\begin{lem}
[\cite{MV1}]
For given data\ {} $q, \bs \La$, if two populations intersect, then they coincide.
\end{lem}

\begin{lem}
[\cite{MV1}]
\label{Weyl LEM}
Let $\mc P\subset \PCr$ be the population generated from the tuple $\bs y$ representing
a critical point $p\in\C^k$ of the master function $\Phi(\cdot,\bs z(q), \bs\La, \bs k)$.
Let $\tilde{\bs y}=(\tilde y_1,\dots,\tilde y_r)$ be a point of $\mc P$ and
$\tilde{\bs k}=(\tilde k_1,\dots,\tilde k_r)$, where $\tilde k_j=\deg \tilde y_j$.
Then the vector $\La_\infty(\bs\La,\tilde{\bs k})$ lies in the orbit
of the vector $\La_\infty(\bs\La,\bs k)$ under the shifted action of the Weyl group.
Conversely, if a vector $\sum_{a=1}^n\La_a-\sum_{j=1}^r\tilde k_j\al_j$ lies in the orbit of
$\La_\infty(\bs\La,\bs k)$, then there exists a tuple $\tilde{\bs y}\in \mc P$ with degree vector
$(\tilde k_1,\dots,\tilde k_r)$.

\end{lem}

\begin{lem}
\label{pop eqn}
Let $\mc P$ be a population associated with data $q,\bs\La$. Let $\bs y \in \mc P$ be a generic fertile tuple.
Then the numbers $\mu_1,\dots,\mu_n$ in equation \Ref{new BAE}
satisfied by $\bs y$  do not depend on the choice of $\bs y$ in $\mc P$.
\end{lem}

\begin{proof} The lemma follows from Lemma \ref{new desc}.
\end{proof}

\begin{rem}
For the Lie algebra $\frak{sl}_2$,  different populations with the same polynomial
$T_1(x)$  have different  sets of numbers  $\mu_1,\dots,\mu_n$.
For the Lie algebra $\frak{sl}_3$, a  population is not uniquely determined by
polynomials $T_1, T_2$ and numbers $\mu_1,\dots,\mu_n$.
For example, let  $(\La_1,\La_2)=(\al_1+\al_2,\al_1+\al_2)$,
$(k_1,k_2)=(1,1)$,  $(z_1,z_2)=(1,-1)$. Then $T_1=T_2=x^2-1$.
The critical point equations
$1/(t-1)+1/(t+1)+1/(t-s)=0,$  $1/(s-1)+1/(s+1)+1/(s-t)=0$ have two solutions
$t=1/\sqrt 5, s=-1/\sqrt 5$ and $t=-1/\sqrt 5, s=1/\sqrt 5$. These solutions generate different populations, see \cite[Section 5]{MV1}.
For the first of them we have $(y_1,y_2) = (x-1/\sqrt 5, x+1/\sqrt 5)$ and for the second
 $(y_1,y_2) = (x+1/\sqrt 5, x-1/\sqrt 5)$. For both of them we have
 \bea
 2 \frac{y_1''}{y_1} + 2\frac{y_2''}{y_2} - 2 \frac{y_1'y_2'}{y_1 y_2}-2\frac{T_1'y_1�}{T_1 y_1} - 2 \frac{T_2' y_2'}{T_2 y_2} +
 \frac {5}{x-1} -\frac{5}{x+1}=0
\eea
and for both of them the numbers $\mu_1,\mu_2$ are the same.

\end{rem}

Given data $q,\bs\La$, let $T_j(x)$ be polynomials defined by \Ref{T}.
Introduce a quadratic polynomial in variables $k_1,\dots,k_r$,
\bean
\label{new BAE form}
B(k_1,\dots,k_r)=
\sum_{j=1}^r(\al_j,\al_j)k_j(k_j-1-\tau_j) + \sum_{i\ne j}(\al_i,\al_j)k_ik_j  .
\eean
We have
\bean
\label{B=( )}
&& B(k_1,\dots,k_r)=
\\
\notag
&&
\phantom{aaa}
=( \rho +\sum_{a=1}^n\La_a-\sum_{j=1}^rk_j\al_j,\ \rho +     \sum_{a=1}^n\La_a -\sum_{j=1}^rk_j\al_j)
 - (\rho+\sum_{a=1}^n\La_a, \ \rho + \sum_{a=1}^n\La_a).
\eean

\begin{lem}
Let $\mc P$ be a population generated from a tuple, which is generic and fertile  with respect to the data $q,\bs \La$.
Then there exists an integer $c(\mc P)$ such that for any $\bs y=(y_1(x),\dots,y_r(x))\in\mc P$
with $\deg y_j = k_j$, we have
\bean
\label{B=c}
B(k_1,\dots,k_r)=c(\mc P).
\eean
\end{lem}

\begin{proof}
By Lemma \ref{pop eqn}, there exists numbers $\mu_1,\dots,\mu_n, \mu_1+\dots+\mu_n=0$, such that
all generic fertile tuples $\bs y\in\mc P$ satisfy equation \Ref{new BAE} with numbers $\mu_1,\dots,\mu_n$.
Equating to zero the coefficient of $x^{-2}$ in the Laurent expansion at infinity of the left hand side of  \Ref{new BAE}
we prove \Ref{B=c} for $c(\mc P)$ equal to the coefficient
of $x^{-2}$ in the Laurent expansion at infinity of $-\sum_{a=1}^n\frac{1}{x-z_a(q)}(\mu_a - \sum_{b\ne a}\frac{(\La_a,\La_b)}{z_a(q)-z_b(q)})$.
\end{proof}

The integer $c(\mc P)$ will be called the {\it charge} of the population $\mc P$.

Let $\mc P$ be a population associated with data $q, \bs\La$.
A tuple ${\bs y}\in \mc P$ with degree vector ${\bs k}=( k_1,\dots,k_r)$ will be  called
{\it minimal} if
$ \tau_j + 1- k_j - \sum_{i\ne j} a_{j,i} k_i > k_j$ for $ j=1,\dots,r$.
These inequalities can be rewritten as
\bean
\label{dominant}
 \tau_j + 1 - \sum_{i=1}^r a_{j,i}  k_i > 0 , \qquad j=1,\dots,r .
\eean

 \begin{lem}

Every population  has a minimal  tuple.
\end{lem}

\begin{lem}
A tuple  ${\bs y}\in \mc P$ is minimal  if and only if the vector
$\La_\infty(\bs \La, {\bs k})$ is integral dominant.

\end{lem}

\begin{lem}
\label{ineq}
Let $\mc P$ be a population associated with data $q,\bs\La$. Let $\bs y\in \mc P$ be a minimal tuple
with degree vector $\bs k=(k_1,\dots,k_r)$. Then either  $\bs k=0$ and $c(\mc P)=0$ or
$\bs k\ne 0$, $c(\mc P)<0$ and
\bean
\label{charge min}
c(\mc P)+\sum_{j=1}^r b_j(\tau_j+1)k_j < 0.
\eean

\end{lem}

\begin{proof} If $\bs k=0$, then $c(\mc P)=0$.
If $\bs k\ne 0$, then $\sum_{j=1}^rb_j  k_j (\tau_j + 1 - \sum_{i=1}^r a_{j,i}  k_i) > 0$.
Hence
\bea
 0 &<& \sum_{j=1}^rb_j  k_j (\tau_j + 1 - \sum_{i=1}^r a_{j,i}   k_i) =
\\
&=& \sum_{j=1}^rb_j  k_j (\tau_j + 1 - \sum_{i=1}^r a_{j,i}   k_i)+ B( k_1,\dots, k_r)- c(\mc P)
 = -c(\mc P)- \sum_{j=1}^r b_j k_j (\tau_j + 1) .
\eea
\end{proof}

The tuple $\bs y^\emptyset=(1,\dots,1)$ is generic and fertile with respect to any data $q,\bs\La$.
 Denote by $\mc P_{y^\emptyset,q,\bs \La}\subset \PCr$ the population of tuples generated from
$\bs y^\emptyset$. Clearly we have $c(\mc P_{\bs y^\emptyset, q,\bs\La})=0$.

\begin{thm}
\label{thm c=0}
The population $\mc P_{\bs y^\emptyset, q,\bs\La}$ is the only population $\mc P$ associated with the data $q,\bs\La$ and such that
$c(\mc P)=0$.

\end{thm}

\begin{proof}
 Let $\mc P$ be any such population. Let $\bs y\in \mc P$ be a minimal tuple and $\bs k$  its degree vector. By Lemma \ref{ineq}, we get $\bs k=0$.
Hence $\mc P$ contains $\bs y^\emptyset$ and therefore $\mc P=\mc P_{\bs y^\emptyset, q,\bs\La}$.
\end{proof}

If $n=1$, then $q\in\C$ and $\bs\La=(\La_1)$.

\begin{thm}
\label{thm n=1}
If $n=1$, then $\mc P_{\bs y^\emptyset, q,\bs\La}$ is the only populations associated with data $q,\bs\La$.

\end{thm}

\begin{proof}
Let $n=1$. Let $\mc P$  be a population associated with data $q,\bs\La$ and $\bs y\in\mc P$ a generic fertile tuple. Then
 the numbers $\mu_1,\dots,\mu_n$ in formula \Ref{new BAE} are equal to zero. Hence $c(\mc P)=0$. By Theorem \ref{thm c=0} we get
$\mc P=\mc P_{\bs y^\emptyset, q,\bs\La}$.
\end{proof}


\begin{rem}
In \cite{ScV, MV1, MV2, F}, it is shown that a population $\mc P\subset\PCr$ is a variety isomorphic to
the flag variety of the Kac-Moody algebra $\g(A^t)$ Langlands dual to the Kac-Moody algebra $\g(A)$.
\end{rem}

\begin{rem}

If $\g(A)$ is an affine Lie algebra of type $A_r$ or $A^{(2)}_2$, the population
$\mc P_{\bs y^\emptyset,q,\bs\La=0}$ for $n=1$ was used in
\cite{VW, VWW} to construct rational solutions of the corresponding integrable hierarchy.

\end{rem}

\subsection{Conjecture}

\begin{conj}
Given a Kac-Moody algebra $\g(A)$ and data $q,\bs \La$, the populations $\mc P$ are in one-to-one correspondence with characters $\chi : \B\to \C$
of the Bethe algebra $\B$ of the Gaudin model associated with the tensor product $\otimes_{a=1}^n L_{\La_a}$
of the irreducible $\g(A)$-modules  $L_{\La_a}$ with highest weights $\La_a$.
\end{conj}


The definition of the Bethe algebra of the Gaudin model see in \cite{FFR, T, MTV}.

The statement of the conjecture for $\g(A)=\frak{sl}_{r+1}$
is a corollary of the main theorem in \cite{MTV} and the description of  $\frak{sl}_{r+1}$-populations $\mc P\subset \PCr$ in \cite[Section 5]{MV1}.
\medskip

Conjecture also holds for $n=1$. Indeed, in  this case, the Bethe algebra $\B$ associated with
an irreducible highest weight
representation $L_{\La_1}$ consists of scalar operators on $L_{\La_1}$ and is isomorphic to $\C$. The Bethe algebra has the single character
$ id : \C\to\C$ and the conjecture says that there is only one population associated with $\g(A)$ and the date $q\in\C$, $\bs \La=(\La_1)$.
Our  Theorem \ref{thm n=1} says that indeed there is exactly one population and
this is the population  $\mc P_{\bs y^\emptyset, q,\bs\La}\subset \PCr$ generated from the fertile generic
 tuple $\bs y^\emptyset=(1,\dots,1)$ representing the critical point of the master function with no variable and associated with $\bs k=(0,\dots,0)$.



\begin{thebibliography}{ccc}




\bibitem[FFR]{FFR}  B. Feigin, E. Frenkel, and N. Reshetikhin,
{\it  Gaudin model, Bethe Ansatz and Critical Level},
Commun. Math. Phys. {\bf 166} (1994),29--62.




\bibitem[F]{F}  E. Frenkel,
{\it Opers on the projective line, flag manifolds and Bethe anzatz},
 Mosc. Math. J. 4 (2004), no. 3, 655�705, 783





\bibitem[K]{K} V. Kac, {\it Infinite-dimensional Lie algebras},
  Cambridge University Press, 1990.

\bibitem[MSTV]{MSTV} E. Mukhin, V. Schechtman, V. Tarasov, A. Varchenko,
{\it On a new form of Bethe ansatz equations and separation of variables in the $sl_3$ Gaudin model},
Proc. Steklov Inst. Math. 258 (2007), no. 1, 155--177


\bibitem[MTV]{MTV} E. Mukhin,  V. Tarasov, A. Varchenko,
{\it Schubert calculus and representations of the general linear group},
J. Amer. Math. Soc. {\bf } (2009), no. 4, 909--940


\bibitem[MV1]{MV1} E. Mukhin and A. Varchenko,
{\it Critical Points of Master Functions and Flag Varieties},
math.QA/0209017 (2002), 1--49.



\bibitem[MV2]{MV2} E. Mukhin and A. Varchenko,
{\it Populations of solutions of the XXX Bethe equations associated to
Kac-Moody algebras}, math.QA/0212092 (2002), 1--8.





\bibitem[ScV]{ScV} I. Scherbak and A. Varchenko, {\it Critical point of
    functions, $sl_2$ representations and Fuchsian differential
    equations with only univalued solutions}, math. QA/0112269, (2001)
    1--25.



\bibitem[SV]{SV} V. Schechtman and A. Varchenko, {\it Arrangements of
    hyperplanes and Lie algebra homology}, Invent. Math., {\bf 106}
    (1991), 139--194.


\bibitem[T]{T} D. Talalaev, {\it Quantization of the Gaudin System}, Preprint  (2004), 1--19, {\tt hep-th/0404153}


\bibitem[VWW]{VWW} A. Varchenko,  T. Woodruff, D. Wright,  {\it     Critical points of master functions and the mKdV hierarchy of type $A^2_2$},
Preprint (2013), 1--42, {\tt   arXiv:1305.5603}


\bibitem[VW]{VW} A. Varchenko, D. Wright, {\it Critical points of master functions and integrable hierarchies},
Preprint (2012), 1--42, {\tt  arXiv:1207.2274}




\end{thebibliography}
\end{document}